\documentclass[11pt,reqno]{amsart}
\usepackage{geometry}
\usepackage{color}

\usepackage{amsfonts,amscd}
\usepackage{amssymb}
\usepackage{url}
\usepackage[english]{babel}
\usepackage{enumitem}
\usepackage{booktabs}
\usepackage{hyperref}
\usepackage{multirow}
\usepackage{afterpage}
\usepackage{float}

\theoremstyle{plain}
\newtheorem{theorem}                 {Theorem}      [section]

\newtheorem{corollary}    [theorem]  {Corollary}

\newtheorem{proposition}  [theorem]  {Proposition}

\theoremstyle{definition}

\newtheorem{definition}   [theorem]  {Definition}

\newtheorem{notation}     [theorem]  {Notation}

\newtheorem{remark}       [theorem]	 {Remark}

\newtheorem{examplex}[theorem]{Example}
\newenvironment{example}
  {\pushQED{\qed}\examplex}
  {\popQED\endexamplex}

\numberwithin{equation}{section}
\numberwithin{table}{section}

\def \cn{{\mathbb C}}

\def \rn{{\mathbb R}}

\def \A{\mathcal A}

\def\nab#1#2{\hbox{$\nabla$\kern -.3em\lower 1.0 ex
		\hbox{$#1$}\kern -.1 em {$#2$}}}

\def \Re{\mathfrak R\mathfrak e\,}
\def \Im{\mathfrak I\mathfrak m\,}

\def \lb#1#2{[#1,#2]}

\def \g{\mathfrak g}
\def \h{\mathfrak h}

\def \r{\mathfrak r}

\def \GLR#1{\text{\bf GL}_{#1}(\rn)}

\def \glr#1{\mathfrak{gl}_{#1}(\rn)}

\def \tSL2{\widetilde{\text{\bf SL}}_{2}(\rn)}

\def \SO#1{\text{\bf SO}(#1)}

\def \SU#1{\text{\bf SU}(#1)}

\def \Sp#1{\text{\bf Sp}(#1)}

\newcommand{\lrpar}[1]{\left( #1 \right)}
\newcommand{\lrcurl}[1]{\left\{ #1 \right\}}

\DeclareMathOperator{\Aut}{Aut}

\DeclareMathOperator{\Div}{div}

\DeclareMathOperator{\trace}{Tr}

\DeclareMathOperator{\Exp}{Exp}
\DeclareMathOperator{\arsinh}{arsinh}

\DeclareMathOperator{\adj}{adj}

\newcommand{\transp}{\mathrm{T}}

\begin{document}

\title[$r$-Harmonic and Complex Isoparametric Functions]{$r$-Harmonic and Complex Isoparametric Functions\\on the Lie Groups $\rn^m \ltimes \rn^n$ and $\rn^m \ltimes \mathrm{H}^{2n+1}$}

\author{Sigmundur Gudmundsson}
\address{Mathematics, Faculty of Science\\
	University of Lund\\
	Box 118, Lund 221\\
	Sweden}
\email{Sigmundur.Gudmundsson@math.lu.se}

\author{Marko Sobak}
\address{Mathematics, Faculty of Science\\
	University of Lund\\
	Box 118, Lund 221\\
	Sweden}
\email{marko.sobak1@gmail.com}

\begin{abstract}
In this paper we introduce the notion of complex isoparametric functions on Riemannian manifolds.  These are then employed to devise a general method for constructing proper $r$-harmonic functions. We then apply this to construct the first known explicit proper $r$-harmonic functions on the Lie group semidirect products $\rn^m \ltimes \rn^n$ and $\rn^m \ltimes \mathrm{H}^{2n+1}$, where $\mathrm{H}^{2n+1}$ denotes the classical $(2n+1)$-dimensional Heisenberg group. In particular, we construct such examples on all the simply connected irreducible four-dimensional Lie groups.
\end{abstract}

\subjclass[2010]{31B30, 53C43, 58E20}

\keywords{Biharmonic functions, solvable Lie groups}

\maketitle

\section{Introduction}

Biharmonic functions are important in physics. Aside from continuum mechanics and elasticity theory, the biharmonic equation also makes an appearance in two-dimensional hydrodynamics problems involving Stokes flows of incompressible Newtonian fluids.  A comprehensive review of this fascinating history of biharmonic functions can be found in the article \cite{Mel}.

On this subject the literature is vast.  With only very few exceptions, the domains are either surfaces or open subsets of flat Euclidean space, see for example \cite{Bai-Far-Oua}. The development of the very last years has changed this and can be traced e.g. in the following publications:  \cite{Gud-13}, \cite{Gud-14}, \cite{Gud-15}, \cite{Gud-Mon-Rat-1}, \cite{Gud-Sif-1}, \cite{Gud-Sif-2}, \cite{Gud-Sob-1}.  There the authors develop methods for constructing explicit  $r$-harmonic functions on the classical Lie groups and even some symmetric spaces.
\smallskip

In this paper we introduce the notion of complex isoparametric functions on a Riemannian manifold $(M,g)$, see Definition \ref{def-isoparametric}.  It turns out that together with the so called eigenfunctions they provide us with a method for manufacturing complex-valued proper $r$-harmonic functions on $(M,g)$, see Section \ref{section-method}.

We then apply our new method to construct proper $r$-harmonic functions on the solvable Lie group semidirect products $\rn^m \ltimes \rn^n$ and $\rn^m \ltimes \mathrm{H}^{2n+1}$, where $\mathrm{H}^{2n+1}$ denotes the classical $(2n+1)$-dimensional Heisenberg group.
The study of these particular Lie groups is motivated by the fact that all four-dimensional irreducible Lie groups are, up to isomorphism, semidirect products of one of these two types.

\section{Preliminaries}\label{section-preliminaries}

Let $(M,g)$ be a smooth manifold equipped with a Riemannian metric $g$.  We complexify the tangent bundle $TM$ of $M$ to $T^{\cn}M$ and extend
the metric $g$ to a complex bilinear form on $T^{\cn}M$.  Then the
gradient $\nabla \phi$ of a complex-valued function $\phi:(M,g)\to\cn$ is a
section of $T^{\cn}M$.  In this situation, the well-known linear
{\it Laplace-Beltrami} operator (alt. {\it tension} field) $\tau$ on $(M,g)$
acts locally on $\phi$ as follows
$$
\tau(\phi)=\Div (\nabla \phi)=  \frac{1}{\sqrt{\det g}} \sum_{ij}
\frac{\partial}{\partial x_j}\left(g^{ij}\, \sqrt{\det g}\,
\frac{\partial \phi}{\partial x_i}\right).
$$
For two complex-valued functions $\phi,\psi:(M,g)\to\cn$ we have the
following well-known product rule
\begin{equation}\label{eq-product-rule}
\tau(\phi\cdot\psi)=\tau(\phi)\cdot\psi+2\,\kappa(\phi,\psi)+\phi\cdot\tau(\psi),
\end{equation}
where the {\it conformality} operator $\kappa$ is given by
$$
\kappa(\phi,\psi)=g(\nabla \phi,\nabla \psi).
$$
Moreover, if $f : U \subset \cn \to \cn$ is a holomorphic function defined on an open set $U$ containing $\phi(M)$, then we have the chain rule
\begin{equation}\label{eq-chain-rule}
\tau(f \circ \phi) = \kappa(\phi,\phi) \, f''(\phi) + \tau(\phi) \, f'(\phi).
\end{equation}

For a positive integer $r$ the iterated Laplace-Beltrami operator
$\tau^r$ is defined inductively by
$$
\tau^{0} (\phi)=\phi,\quad \tau^r (\phi)=\tau(\tau^{(r-1)}(\phi)).
$$
\begin{definition}\label{definition-proper-p-harmonic}
Let $r$ be a positive integer. Then a smooth complex-valued function $\phi:(M,g)\to\cn$ is said to be
\begin{enumerate}
\item[(a)] {\it $r$-harmonic} if $\tau^r (\phi)=0$,
\item[(b)] {\it proper $r$-harmonic} if $\tau^r (\phi)=0$ and
$\tau^{(r-1)}(\phi)$ does not vanish identically.
\end{enumerate}
\end{definition}

It should be noted that the {\it harmonic} functions are exactly the
$1$-harmonic and the {\it biharmonic} functions are the $2$-harmonic
ones. In some texts the $r$-harmonic functions are also called {\it polyharmonic} of order $r$.
We also note that if a function is $r$-harmonic then it is also $s$-harmonic for any $s \geq r$.
Hence, one is usually interested in studying functions which are {\it proper} $r$-harmonic.

\section{Complex Isoparametric Functions}\label{section-method}

A method for constructing proper $r$-harmonic functions on Riemannian manifolds has recently been developed in \cite{Gud-Sob-1}.
The $r$-harmonic functions that the authors consider are of the form $f \circ \phi$, where $f : U \subset \cn \to \cn$ is a holomorphic function 
defined on an open set $U$ containing $\phi(M)$, and
$\phi: (M,g) \to \cn$ is an {\it eigenfunction}
i.e.\ a smooth complex-valued function such that
\begin{equation}\label{eq-eigenfunction-def}
\tau(\phi) = \lambda \cdot \phi \quad\text{and}\quad \kappa(\phi,\phi) = \mu \cdot \phi^2,
\end{equation}
for some constants $\lambda, \mu \in \cn$.
The construction from \cite{Gud-Sob-1} can be adapted to the more general setting when $\phi$ is a complex isoparametric function, as we will now demonstrate.

Classically, isoparametric functions on Riemannian manifolds are defined as real-valued functions $\phi : (M,g) \to \rn$ such that the tension field $\tau$ and the conformality operator $\kappa$ satisfy
\begin{equation}\label{eq-real-isoparametric-def} 
\tau(\phi) = \Phi \circ \phi \quad\text{and}\quad \kappa(\phi, \phi) = \Psi \circ \phi,
\end{equation}
for some smooth functions $\Phi,\Psi$.
These have been extensively studied due to their beautiful geometric properties, see e.g.\ \cite{Tho}.
As we are mainly interested in complex-valued functions, the following complex-valued analogue of the classical real-valued iso\-parametric functions will turn out to be useful. 

\begin{definition}\label{def-isoparametric}
Let $(M,g)$ be a Riemannian manifold.
Then a smooth complex-valued function $\phi : M \to \cn$ is said to be \textit{complex isoparametric} on $M$ if 
there exist holomorphic functions $\Phi,\Psi : U \to\cn$ defined on some open set $U\subset \cn$ containing $\phi(M)$, such that
the tension field $\tau$ and the conformality operator $\kappa$ satisfy
\begin{equation}\label{eq-complex-isoparametric-def}
\tau(\phi) = \Phi \circ \phi \quad\text{and}\quad \kappa(\phi, \phi) = \Psi \circ \phi.
\end{equation}
\end{definition}

\begin{remark}
Note that the term \textit{complex isoparametric} is used here mainly for aesthetical reasons, as the defining conditions (\ref{eq-real-isoparametric-def}) for real isoparametric functions and (\ref{eq-complex-isoparametric-def}) for complex ones are identical.
Despite this similarity, there does not seem to be a direct relation between these two concepts.
E.g.\ for a complex-valued function $\phi: M \to \cn$ to be complex isoparametric, it is neither necessary nor sufficient that its real and imaginary parts are real isoparametric.
This is caused not only by the restrictive condition of $\Phi$ and $\Psi$ being holomorphic, but also by the fact that the relation for the conformality operator may fail to be satisfied, as demontrated by the functions 
$$\rn \ni t \mapsto e^{(a+\mathrm ib)t} \quad\text{and}\quad \rn \ni t \mapsto e^{at} + \mathrm ie^{bt},\qquad a,b\in \rn.$$
Indeed, the former is complex isoparametric but its real and imaginary parts are not real isoparametric, while the latter is not complex isoparametric when $a\not= b$ but its real and imaginary parts are real isoparametric.
\end{remark}

\begin{remark}
We also note that other different definitions of complex-valued isoparametric functions have been proposed.
For example, in \cite{Bai} the author defines a function $\phi:M \to \cn$ to be complex isoparametric if
\begin{equation*}
\tau(\phi) = \Phi \circ \phi \quad\text{and}\quad |d\phi|^2 = \kappa(\phi,\bar\phi) = \Psi \circ \phi
\end{equation*}
for a smooth complex-valued function $\Phi$ and a smooth real-valued function $\Psi$. 
This definition is less restrictive than ours due to the weaker regularity assumptions on $\Phi$ and $\Psi$, but it is also essentially different because of the assumed condition for the conformality operator $\kappa$. Note that the value of $\kappa(\phi,\phi)$ contains no information about the value of $\kappa(\phi,\bar\phi)$, and vice versa.
\end{remark}

A plethora of examples of eigenfunctions (which are a fortiori complex isoparametric) on classical semisimple Lie groups can be found in Table 1 of \cite{Gud-Sob-1}.
Examples of complex isoparametric functions which are not eigenfunctions shall be studied later in this work, see Proposition \ref{lemma-phi}.

\medskip

As in \cite{Gud-Sob-1}, the gist of constructing $r$-harmonic functions using 
isoparametric functions is to study compositions of the form $f \circ \phi$,
where $\phi : (M,g) \to \cn$ is complex isoparametric and $f: U \to \cn$ is a holomorphic function defined on some open set $U \subset \cn$ containing $\phi(M)$.
By the chain rule (\ref{eq-chain-rule}), the tension field of such compositions satisfies
\begin{equation*}
\tau(f \circ \phi) = \kappa(\phi,\phi)\, f''(\phi) + \tau(\phi)\, f'(\phi) = (\Psi \, f'' + \Phi \, f') \circ \phi.
\end{equation*}
With this in hand, one can first obtain a harmonic function $f_1 \circ \phi$ by solving
\begin{equation*}
\tau(f_1 \circ \phi) = 0,
\end{equation*}
which, by the chain rule, reduces to the complex ordinary differential equation
\begin{equation*}
\Psi\,f_1'' + \Phi\,f_1' = 0.
\end{equation*}
Continuing inductively, we assume that a proper $(r-1)$-harmonic function of the form $f_{r-1} \circ \phi$ has already been constructed.
We can then obtain a proper $r$-harmonic function $f_r \circ \phi$ by solving the problem
\begin{equation*}
\tau(f_r \circ \phi) = f_{r-1} \circ \phi,
\end{equation*}
which again reduces to a complex ordinary differential equation, namely
\begin{equation*}
\Psi\, f_r'' + \Phi\, f_r' = f_{r-1}.
\end{equation*}
More explicitly, we get the following result.

\begin{theorem}\label{thm-p-harmonic-main}
Let $(M,g)$ be a Riemannian manifold and $\phi : M \to \cn$ be a complex isoparametric function on $M$ with
\begin{equation*}
\tau(\phi) = \Phi \circ \phi \quad\text{and}\quad \kappa(\phi,\phi) = \Psi \circ \phi,
\end{equation*}
for some holomorphic functions $\Phi,\Psi : U \to\cn$ defined on an open set $U$ of $\cn$ containing $\phi(M)$.
Suppose that one of the following situations holds.
\begin{enumerate}
\item If $\Psi$ vanishes identically, let  $\hat{U}$ be an open simply connected subset of \newline ${U \setminus \Phi^{-1}(\{0\})}$ 
and define the holomorphic functions ${f_r : \hat{U} \to \cn}$ for $r \geq 1$ by
\begin{equation*}
f_r(z) = c \lrpar{\int^z \frac{\mathrm d\zeta}{\Phi(\zeta)}}^{r-1},
\end{equation*}
where $c \in \cn$ is non-zero.
\item If $\Psi$ does not vanish identically, let $\hat{U}$ be an open simply connected subset of $U \setminus \Psi^{-1}(\{0\})$, put
\begin{equation*}
\Lambda(z) = \exp\lrpar{-\int^z \frac{\Phi(\zeta)}{\Psi(\zeta)} \,\mathrm d\zeta}, \quad z \in \hat{U}
\end{equation*}
and define the holomorphic functions ${f_r : \hat{U} \to \cn}$ for $r \geq 1$ by
\begin{align*}
f_1(z) &= c_1 \int^z \Lambda(\zeta) \, \mathrm d\zeta + c_2,\\[0.1cm]
f_r(z) &= \int^z \Lambda(\eta) \int^\eta  \frac{f_{r-1}(\zeta)}{\Lambda(\zeta)\,\Psi(\zeta)} \, \mathrm d\zeta \,\mathrm d\eta, \quad r > 1,
\end{align*}
where $c = (c_1, c_2) \in \cn^2$ is non-zero.
\end{enumerate}
Then in both cases, the composition
\begin{equation*}
f_r \circ \phi : \phi^{-1} (\hat{U}) \to \cn
\end{equation*}
is proper $r$-harmonic on its open domain $\phi^{-1} (\hat{U})$ in $M$ for all $r \geq 1$. 
\end{theorem}

\begin{remark}
If the isoparametric function $\phi$ is real-valued, then one can weaken the assumptions in Theorem \ref{thm-p-harmonic-main} 
by only requiring that $\Phi$ and $\Psi$ are smooth functions of a real variable and $\hat{U}$ can be taken to be an interval. In the complex-valued case, the requirement that $\hat{U}$ is simply connected  is needed to ensure that the holomorphic antiderivatives are well-defined.
\end{remark}

Upon applying Theorem \ref{thm-p-harmonic-main} to the particular case when the isoparametric function $\phi$ is an eigenfunction on $(M,g)$ 
satisfying (\ref{eq-eigenfunction-def}), one sees that the composition $f_r \circ \phi$ is $r$-harmonic if
\begin{equation}\label{eq-p-harmonic-eigenfunction}
f_r(z)= 
\begin{cases}
c\log(z)^{r-1} 																& \text{if }\; \mu = 0, \; \lambda \not= 0\\[0.2cm]
c_1\log(z)^{2r-1}+ c_{2}\log(z)^{2r-2}			 							& \text{if }\; \mu \not= 0, \; \lambda = \mu\\[0.2cm]
c_1 z^{1-\frac\lambda{\mu}}\log(z)^{r-1} + c_2 \log(z)^{r-1}				& \text{if }\; \mu \not= 0, \; \lambda \not= \mu
\end{cases}
\end{equation}
which has already been obtained in the aforementioned paper \cite{Gud-Sob-1}.

\section{The Semidirect Products $\rn^m \ltimes \rn^n$ and $\rn^m \ltimes \mathrm{H}^{2n+1}$}

Low-dimensional Lie groups, particularly of dimension three and four, are of great importance in physics.
Probably most notably such Lie groups are used as models of spacetime in the theory of general relativity.
By Lie's third theorem, there is a bijection between real finite-dimensional Lie algebras and connected simply connected Lie groups of the same dimension.
Conveniently, real Lie algebras of dimension less than or equal to six have been classified,
see e.g.\ the recent book \cite{Sno-Win}.
As a consequence, one also obtains classifications of low-dimensional connected simply connected Lie groups.

The study conducted in this paper was initially meant to be a study of $r$-harmonic functions on the four-dimensional connected simply connected Lie groups.
For our purposes, only the Lie groups whose Lie algebras are indecomposable, i.e.\ not direct products of lower dimensional Lie algebras, are of interest.
The reason for this is that the Lie groups whose Lie algebras are decomposable are themselves direct products of lower dimensional Lie groups,
and the theory of $p$-harmonic functions on product manifolds is well-known, see e.g.\ \cite{Gud-14}.

According to the classification given for example in \cite{Pop-Boy-etal} and \cite{Big-Rem}, all four-dimensional indecomposable real Lie algebras are semidirect products of one of the following types:
\begin{equation*}
\r \ltimes \r^3, \quad \r^2 \ltimes \r^2, \quad \r \ltimes \h^{3},
\end{equation*}
where $\r^n$ denotes the $n$-dimensional abelian algebra and $\h^{2n+1}$ denotes the $(2n+1)$-dimensional Heisenberg algebra.
For the reader's convenience we list these Lie algebras in Table \ref{table:4dimalgebras}.
We thus see that the corresponding four-dimensional connected simply connected Lie groups are semidirect products of the form
\begin{equation*}
\rn \ltimes \rn^3, \quad \rn^2 \ltimes \rn^2, \quad \rn \ltimes \mathrm{H}^{3},
\end{equation*}
where $\mathrm{H}^{2n+1}$ is the $(2n+1)$-dimensional Heisenberg group.
Linear representations of these four-dimensional Lie groups can be found in Table \ref{table-groups-4dim}.
This motivates our interest in studying the more general semidirect products $$\rn^m \ltimes \rn^n \quad\text{and}\quad \rn^m \ltimes \mathrm{H}^{2n+1}.$$
In particular, we note that such semidirect products are automatically solvable and hence diffeomorphic to the vector space of the corresponding dimension.

\subsection{The Semidirect Products $\rn^m \ltimes \rn^n$}

Let $\mu : \rn^m \to \Aut(\rn^n) = \GLR n$ be the smooth homomorphism
\begin{equation*}
\mu(t) = \Exp \lrpar{ \sum_{k=1}^m A_kt_k },
\end{equation*}
for some family $\A = (A_k)_{k=1}^m$ of commuting matrices in $\glr n$.
Then the semidirect product $\rn^m \ltimes_\mu \rn^n$ is by definition
the smooth manifold $\rn^m \times \rn^n$ equipped with the Lie group operation
\begin{equation*}
(t,x)(s,y) = (t+s, \, x+\mu(t)y), \quad (t,x),(s,y) \in \rn^m \times \rn^n.
\end{equation*}
Note that in this case, the family $\A$ completely determines the semidirect product $\rn^m \ltimes_\mu \rn^n$,
so that it is natural to use the more suggestive notation $\rn^m \ltimes_\A \rn^n$.

The Lie group $\rn^m \ltimes_\A \rn^n$ has a faithful linear representation as the matrix group
\begin{equation*}
\big\{ \begin{bmatrix}
\mu(t) & x & 0\\
0 & 1 & 0 \\
0 & 0 & \Exp(\sum_{k} D_k t_k)
\end{bmatrix} \,\mid\, (t,x) \in \rn^m \times \rn^n \big\},
\end{equation*}
where $(D_k)_{ij} = \delta_{ik}\delta_{jk}$.
The $m$-dimensional block
\begin{equation*}
\Exp\lrpar{\sum_{k=1}^m D_k t_k}
\end{equation*}
of the representation is needed in general since $\mu$ may not be injective, but (parts of) this block may often be omitted.
This linear representation induces a natural basis for the corresponding Lie algebra $\r^m \ltimes_\A \r^n$, namely that  consisting of the elements
\begin{equation*}
\frac{\partial}{\partial t_k}\bigg|_0
=
\begin{bmatrix}
A_k & 0 & 0 \\
0 &  0 & 0\\
0 & 0 & D_k
\end{bmatrix},
\quad
\frac{\partial}{\partial x_i}\bigg|_0
=
\begin{bmatrix}
0 & e_i & 0\\
0 & 0 & 0 \\
0 & 0 & 0
\end{bmatrix},
\end{equation*}
where $1 \leq k \leq m, \; 1 \leq i \leq n$, and $e_i$ are the canonical unit vectors in $\rn^n$.

We now define an inner product on the Lie algebra $\r^m \ltimes_\A \r^n$ by requiring that this basis becomes orthonormal and we extend this inner product to a left-invariant Riemannian metric on the Lie group $\rn^m \ltimes_\A \rn^n$.
A simple calculation shows that, in the coordinates $(t,x) \in \rn^m \times \rn^n$, this metric is given by
\begin{equation*}
g_{(t,x)} = \begin{bmatrix}
I_m & 0 \\
0 & \mu(-t)^\transp \mu(-t)
\end{bmatrix}.
\end{equation*}

It is easy to see that if $\phi, \psi : (\rn^m \ltimes_\A \rn^n, g) \to \cn$ are two complex-valued functions then the Laplace-Beltrami operator satisfies
\begin{equation}\label{eq-rm*rn-lap}
\tau(\phi)
=  \sum_{k=1}^m\lrpar{\frac{\partial^2\phi}{\partial t_k^2} - \trace(A_k) \,\frac{\partial\phi}{\partial t_k}}
+ \sum_{i,j=1}^n (\mu(t)\,\mu(t)^\transp)_{ij} \, \frac{\partial^2 \phi}{\partial x_i \partial x_j},
\end{equation}
while the conformality operator is given by
\begin{equation}\label{eq-rm*rn-conf}
\kappa(\phi, \psi) = \sum_{k=1}^m \frac{\partial\phi}{\partial t_k}\frac{\partial\psi}{\partial t_k}
+ \sum_{i,j=1}^n  (\mu(t)\,\mu(t)^\transp)_{ij} \frac{\partial\phi}{\partial x_i}\frac{\partial\psi}{\partial x_j}.
\end{equation}

\medskip

\subsection{The Semidirect Products $\rn^m \ltimes \mathrm{H}^{2n+1}$}

Throughout this section we let $J_{2n}$ denote the block diagonal $2n \times 2n$ matrix defined by
\begin{equation*}
J_{2n} x = (-x_2, x_1, \ldots, -x_{2n}, x_{2n-1}).
\end{equation*}
Note that this is the standard complex structure on $\rn^{2n} \cong \cn^n$.

We recall that the Heisenberg group $\mathrm{H}^{2n+1}$ can be seen as the space $\rn \times \rn^{2n}$ equipped with the Lie group operation
\begin{equation*}
(\xi, x) \boxplus (\eta, y) = (\xi + \eta + \tfrac{1}{2}\langle J_{2n} x, y \rangle, \, x+y).
\end{equation*}
Its Lie algebra $\h^{2n+1}$ is nilpotent and has one-dimensional center.
More explicitly for a basis $\{\Xi, X_1, \ldots, X_{2n}\}$ of $\h^{2n+1}$ the non-zero Lie brackets are given by
\begin{equation*}
[X_{2i-1},X_{2i}] = \Xi, \quad 1\leq i \leq n.
\end{equation*}

In what follows, we consider semidirect products $\rn^m \ltimes_{\hat\mu} \mathrm{H}^{2n+1}$ with respect to a specific class of homomorphisms
 $\hat\mu : \rn^m \to \Aut(\mathrm{H}^{2n+1})$.
The reader interested in more details about this can find them in \cite{Sob-MSc}.
Let $\hat\mu : \rn^m \to \Aut(\mathrm{H}^{2n+1})$ be the smooth homomorphism
\begin{equation*}
\hat\mu(t) = \begin{bmatrix}
a(t) & 0 \\
0 & \mu(t)
\end{bmatrix},
\end{equation*}
where $a : \rn^m \to \rn$ and $\mu : \rn^m \to \GLR{2n}$ are given by
\begin{equation*}
a(t) =  \exp\lrpar{\frac{1}{n}\sum_{k=1}^m \trace(A_k)t_k} \quad\text{and}\quad \mu(t) = \Exp\lrpar{\sum_{k=1}^m A_kt_k},
\end{equation*}
for some family $\A = (A_k)_{k=1}^m$ of commuting matrices in $\glr{2n}$ of the form
\begin{equation}\label{eq-heisenberg-der-A}
{\small
\begin{bmatrix}
A_{(1,1)} & -\adj(A_{(2,1)}) & - 	\adj(A_{(3,1)}) & \ldots & -\adj(A_{(n,1)})\\
A_{(2,1)} & A_{(2,2)} & -\adj(A_{(3,2)}) & \ldots & -\adj(A_{(n,2)})\\
A_{(3,1)} & A_{(3,2)} & A_{(3,3)} & \ldots & -\adj(A_{(n,3)}) \\
\vdots & & & \ddots & \vdots\\
A_{(n,1)} & A_{(n,2)} & A_{(n,3)} & \ldots & A_{(n,n)}
\end{bmatrix}
},
\end{equation}
where $A_{(i,j)} \in \rn^{2\times 2}$ are such that $\trace A_{(i,i)} = a$ for $1 \leq i \leq n$, and $\adj(A_{(i,j)})$ denotes
the adjugate i.e.\ the transpose of the cofactor matrix of $A_{(i,j)}$.
In this situation, the Lie group semidirect product $\rn^m \ltimes_{\hat\mu} \mathrm{H}^{2n+1}$ is the manifold $\rn^m \times \rn \times \rn^{2n}$
equipped with the Lie group operation
\begin{align*}
(t,\xi,x) \, (s,\eta,y) &= \big(t+s, \, (\xi,x)\boxplus(a(t) \eta, \mu(t)y)\big)\\[0.1cm]
&= \big(t+s, \; \xi + a(t)\eta + \tfrac{1}{2}\langle J_{2n} x, \mu(t)y \rangle, \; x + \mu(t)y \big),
\end{align*}
where $(t,\xi,x) , (s,\eta,y) \in \rn^m \times \rn \times \rn^{2n}$.
As in the previous section, such semidirect products are completely determined by the family $\A$,
so that it is natural to use the more suggestive notation $\rn^m \ltimes_\A \mathrm{H}^{2n+1}$.
A faithful linear representation of the Lie group $\rn^m \ltimes_\A \mathrm{H}^{2n+1}$ is given by the matrix group
\begin{equation*}
\big\{\begin{bmatrix}
a(t) & \frac{1}{2}(J_{2n} x)^\transp \mu(t) & \xi & 0\\
0 & \mu(t) & x & 0\\
0 & 0 & 1 & 0 \\
0 & 0 & 0 & \Exp(\sum_k D_kt_k)
\end{bmatrix}
\,\mid\,
(t,\xi,x) \in \rn^m \times \rn \times \rn^{2n}
\big\},
\end{equation*}
where as before $(D_k)_{ij} = \delta_{ik}\delta_{jk}$.
Just as in the previous section, the $m$-dimensional block 
\begin{equation*}
\Exp\lrpar{\sum_{k=1}^m D_k t_k}
\end{equation*}
of the representation is needed in general, but (parts of) it can often be removed.

We equip the Lie group $\rn^m \ltimes_\A \mathrm{H}^{2n+1}$ with the left-invariant Riemannian metric induced by the inner 
product on its Lie algebra $\r^m \ltimes_\A \h^{2n+1}$ defined by requiring that the basis
\begin{equation*}
\small
\frac{\partial}{\partial t_k}\bigg|_0 = \begin{bmatrix}
\frac{1}{n}\trace(A_k) & 0 & 0 & 0\\
0 & A_k & 0 & 0\\
0 & 0 & 0 & 0 \\
0 & 0 & 0 & D_k
\end{bmatrix},
\end{equation*}
\begin{equation*}
\frac{\partial}{\partial \xi}\bigg|_0 = \begin{bmatrix}
0 & 0 & 1 & 0\\
0 & 0_{2n} & 0 & 0\\
0 & 0 & 0 & 0 \\
0 & 0 & 0 & 0_{m}
\end{bmatrix},
\quad
\frac{\partial}{\partial x_i}\bigg|_0 = \begin{bmatrix}
0 & 0 & 0 & 0\\
0 & 0_{2n} & e_i & 0\\
0 & 0 & 0 & 0 \\
0 & 0 & 0 & 0_{m}
\end{bmatrix}
\end{equation*}
on $\r^m \ltimes_\A \h^{2n+1}$ becomes orthonormal.
An elementary but long calculation shows that this metric can be expressed in the coordinates $(t,\xi,x) \in \rn^m \times \rn \times \rn^{2n}$ as
\begin{align*}
g_{(t,\xi,x)} 
&=
\begin{bmatrix}
I_m & 0 & 0 \\
0 & 1 & 0\\
0 & -\frac{1}{2}J_{2n} x & I_{2n}
\end{bmatrix}
\begin{bmatrix}
I_m & 0 & 0\\
0 &  a(-t)^2 & 0\\
0 & 0 & \mu(-t)^\transp \mu(-t)
\end{bmatrix}
\begin{bmatrix}
I_m & 0 & 0 \\
0 & 1 & -\frac{1}{2}(J_{2n} x)^\transp\\
0 & 0 & I_{2n}
\end{bmatrix}\\[0.1cm]
&=
\begin{bmatrix}
I_m & 0 & 0\\
0 & 0 & 0\\
0 & 0 & \mu(-t)^\transp \mu(-t)
\end{bmatrix}
+
a(-t)^2
\begin{bmatrix}
0_m & 0 & 0\\[0.1cm]
0 & 1 & -\frac{1}{2}(J_{2n} x)^\transp\\[0.1cm]
0 & -\frac{1}{2} J_{2n}x & \frac{1}{4}J_{2n} x(J_{2n} x)^\transp
\end{bmatrix}.
\end{align*}

With this in hand, it is easy to see that if $\phi,\psi : (\rn^m \ltimes_\A \mathrm{H}^{2n+1},g) \to \cn$ are two complex-valued functions, then
the Laplace-Beltrami operator satisfies
\begin{align}\label{eq-rm*heisenberg-lap}
\tau(\phi)
&= \sum_{k=1}^m \lrpar{ \frac{\partial^2 \phi}{\partial t_k^2} - \frac{n+1}{n}\,\trace(A_k) \, \frac{\partial \phi}{\partial t_k} }
  + \lrpar{a(t)^2 + \frac{1}{4} \langle \mu(t)^\transp J_{2n} x,\mu(t)^\transp J_{2n} x \rangle} \frac{\partial^2 \phi}{\partial \xi^2}\nonumber\\
& \quad + \sum_{i=1}^{2n} (\mu(t)\,\mu(t)^\transp J_{2n} x)_i \, \frac{\partial^2 \phi}{\partial \xi \partial x_i}
+ \sum_{i,j=1}^{2n} (\mu(t) \, \mu(t)^\transp)_{ij} \, \frac{\partial^2 \phi}{\partial x_i \partial x_j}.
\end{align}
and  the conformality operator satisfies
\begin{align}\label{eq-rm*heisenberg-conf}
\kappa(\phi,\psi) &= \sum_{k=1}^m \frac{\partial \phi}{\partial t_k} \frac{\partial \psi}{\partial t_k}
+ \lrpar{a(t)^2 + \frac{1}{4} \langle \mu(t)^\transp J_{2n} x,\mu(t)^\transp J_{2n} x \rangle } \frac{\partial \phi}{\partial \xi}\frac{\partial \psi}{\partial \xi} \nonumber\\
&+ \frac{1}{2}\sum_{i=1}^{2n} (\mu(t)\,\mu(t)^\transp J_{2n} x)_i
\lrpar{ \frac{\partial \phi}{\partial \xi}\frac{\partial \psi}{\partial x_i} + \frac{\partial \phi}{\partial x_i}\frac{\partial \psi}{\partial \xi} } \\
&+ \sum_{i,j=1}^{2n} (\mu(t) \, \mu(t)^\transp)_{ij} \,\frac{\partial \phi}{\partial x_i} \frac{\partial \psi}{\partial x_j}, \nonumber
\end{align}

\medskip

\section{$r$-Harmonic Functions on $\rn^m \ltimes \rn^n$ and $\rn^m \ltimes \mathrm{H}^{2n+1}$}

As we have seen in the preceding sections,
there are a lot of similarities between the Laplace-Beltrami operators and the conformality operators on the Lie groups
$\rn^m \ltimes \rn^n$ and $\rn^m \ltimes \mathrm{H}^{2n+1}$ so that here we will study these in parallel.
To simplify the statements of our results we fix the following notation.

\begin{notation}
Let $G$ be a Lie group, either $\rn^n$ or $\mathrm{H}^{2n+1}$.
If $G = \rn^n$ the indices $i,j$ are assumed to satisfy $1\le i,j\le n$ and in the case $G = \mathrm{H}^{2n+1}$ the condition $1\le i,j\le 2n$. In both cases we have $1 \leq k \leq m$, unless otherwise specified.
Throughout this section $\A = (A_k)_k$ will denote a commuting family of inveritable real matrices which are of dimensions $n\times n$ if $G = \rn^n$
and of dimensions $2n\times 2n$ if $G = \mathrm{H}^{2n+1}$.
In the latter case, we also assume that each member of the family $\A$ is of the form (\ref{eq-heisenberg-der-A}).
Finally, we will use the following notation
\begin{equation}\label{eq-omega-def}
\rn^m \ni \omega = \begin{cases}
(\trace(A_1),\ldots,\trace(A_m)) & \text{if } G=\rn^n,\\[0.2cm]
\frac{n+1}{n} (\trace(A_1),\ldots,\trace(A_m)) & \text{if } G =\mathrm{H}^{2n+1}. 
\end{cases}
\end{equation}

Now if $\phi,\psi : \rn^m \ltimes_\A G \to \cn$ depend only on the coordinates $t$ and $x$, i.e.\ if they are independent of $\xi$, in the case when $G = \mathrm{H}^{2n+1}$, then the Laplace-Beltrami operator satisfies
\begin{equation*}
\tau(\phi) = \sum_k \lrpar{ \frac{\partial^2 \phi}{\partial t_k^2} - \omega_k\, \frac{\partial \phi}{\partial t_k} }
+ \sum_{ij} (\mu(t)\,\mu(t)^\transp)_{ij} \, \frac{\partial^2 \phi}{\partial x_i \partial x_j}
\end{equation*}
and the conformality operator is given by
\begin{equation*}
\kappa(\phi,\psi)
= \sum_k \frac{\partial\phi}{\partial t_k}\frac{\partial\psi}{\partial t_k} + \sum_{ij} (\mu(t)\,\mu(t)^\transp)_{ij} \,\frac{\partial \phi}{\partial x_i}\frac{\partial \psi}{\partial x_j},
\end{equation*}
regardless of the choice of the Lie group $G$.
\end{notation}

Our first result is now the following one of general variable separation.

\begin{proposition}\label{prop-separation-1}
Let $\phi,\psi : \rn^m \ltimes_\A G \to \cn$ be two complex-valued functions such that $\phi$
depends only on $t \in \rn^m$ while $\tau^\alpha(\psi)$ is independent of $t$ for all $\alpha \geq 0$.
Then the tension field of their product $\phi\cdot\psi$ satisfies the identity
\begin{equation*}
\tau^r(\phi \cdot \psi) = \sum_{\alpha=0}^r \binom{r}{\alpha} \, \tau^\alpha(\phi) \, \tau^{r-\alpha}(\psi).
\end{equation*}
In particular, if $\phi$ and $\psi$ are proper $p$-harmonic and proper $r$-harmonic on $\rn^m \ltimes_\A G$, respectively,
then their product $\phi\cdot\psi$ is proper $(p+r-1)$-harmonic
on $\rn^m \ltimes_\A G$.
\end{proposition}

\begin{remark}
This result is reminiscent of the variable separation statement  on product manifolds, which can be found in Lemma 6.1 of \cite{Gud-14}. For our result here we need the additional assumption that $\tau^\alpha(\psi)$ is independent of $t$ for all $\alpha \geq 0$.
This assumption is superfluous for direct products of manifolds, but it is essential in our cases.
\end{remark}

\begin{proof}
Since $\phi$ is a function of $t$, we see by the formulae (\ref{eq-rm*rn-lap}) and (\ref{eq-rm*heisenberg-lap}) for the 
Laplace-Beltrami operator, that $\tau^\beta(\phi)$ remains to be a function of $t$ for all $\beta \geq 0$.
Since by assumption, the tension field $\tau^\alpha(\psi)$ is independent of $t$ for all $\alpha\geq 0$, it follows from
the relations (\ref{eq-rm*rn-conf}) and (\ref{eq-rm*heisenberg-conf}) for the conformality operator, that
\begin{equation*}
\kappa(\tau^\beta(\phi), \tau^\alpha(\psi)) = 0,
\end{equation*}
for all $\alpha,\beta \geq 0$. The first part of our statement now follows easily by induction combined with the product rule (\ref{eq-product-rule}) for the Laplace-Beltrami operator.
For the final claim of the result notice that the identity we have just proven implies that
\begin{equation*}
\tau^{p+r-2}(\phi\cdot\psi) = \binom{p+r-2}{p-1}\, \tau^{p-1}(\phi) \, \tau^{r-1}(\psi) \not= 0
\quad\text{and}\quad
\tau^{p+r-1}(\phi\cdot\psi) = 0,
\end{equation*}
so that the product $\phi\cdot\psi$ is indeed proper $(p+r-1)$-harmonic.
\end{proof}

Constructing proper $r$-harmonic functions $\psi : \rn^m \ltimes_\A G \to \cn$ such that $\tau^\alpha(\psi)$ is independent of $t$ for $\alpha \geq 0$
seems difficult in general, if not impossible.
However, we note that any  harmonic function on $\rn^m \ltimes_\A G$ which is independent of $t$ trivially satisfies this condition.
On the other hand, we can use our main Theorem \ref{thm-p-harmonic-main} to construct proper $r$-harmonic functions depending only on $t$.
We summarize these thoughts in the following statement.

\begin{proposition}\label{prop-pharmonic-x-t} 
\leavevmode
\begin{enumerate}[label=(\roman*), font=\upshape, itemsep=0.3cm]
\item Define the complex-valued function $\psi : \rn^m \ltimes_\A G \to \cn$ by
\begin{equation*}
\begin{cases}
\psi(x) = a + \textstyle\sum_i v_i x_i + \textstyle\sum_{ij} B_{ij}x_ix_j & \text{if } G= \rn^n,\\[0.2cm]
\psi(\xi,x) = a + b\xi +  \textstyle\sum_i v_i x_i + \textstyle\sum_{ij} B_{ij}x_ix_j & \text{if } G= \mathrm{H}^{2n+1},
\end{cases}
\end{equation*}
where $a,b,v_i,B_{ij}\in \cn$ are not all zero and the coefficients $B_{ij}$ form a symmetric matrix such that
\begin{equation*}
\trace (\mu(t)\,\mu(t)^\transp B) = 0, \quad t \in \rn^m.
\end{equation*}
Then $\psi$ is proper harmonic on $\rn^m \ltimes_\A G$.

\item Let $(r_k)_{k=1}^m$ be a collection of positive integers
and define the complex-valued functions $\phi_{k} : \rn^m \ltimes_\A G\to \cn$ by
\begin{equation*}
\phi_{k}(t_k) = \begin{cases}
c_1 \, t_k^{r_k-1} e^{\omega_k t_k} + c_{2} \, t_k^{r_k-1} & \text{if }\;\omega_k \not= 0,\\[0.2cm]
c_1 \, t_k^{2r_k-1} + c_{2} \, t_k^{2r_k-2} & \text{if }\;\omega_k = 0,
\end{cases}
\end{equation*}
where $c = (c_1, c_2) \in \cn^2$ is non-zero. Then for each $k = 1,\ldots, m$, the function $\phi_{k}$ is proper $r_k$-harmonic on $\rn^m \ltimes_\A G$.
Furthermore, their product
\begin{equation*}
\prod_{k=1}^m \phi_k(t_k)
\end{equation*}
is proper $(r_1 + \ldots + r_m -m+1)$-harmonic on $\rn^m \ltimes_\A G$.
\end{enumerate}
\end{proposition}

\begin{proof}
To prove part (i) we note that since $B$ is symmetric we have
\begin{equation*}
\frac{\partial^2 \psi}{\partial x_i \partial x_j} = 2\,B_{ij}
\end{equation*}
so that
\begin{equation*}
\tau(\psi) = 2\textstyle\sum_{ij} (\mu(t) \, \mu(t)^\transp)_{ij} \,B_{ij} = 2\, \trace (\mu(t)\,\mu(t)^\transp B),
\end{equation*}
confirming the result for both choices of the Lie group $G$.

For part (ii) we note that
\begin{equation*}
\tau(t_k) = -\omega_k  \quad\text{and}\quad \kappa(t_k,t_k) = 1,
\end{equation*}
so that for each $k=1,\ldots,m$ the coordinate function $t_k$ is isoparametric.
The first claim then follows from Theorem \ref{thm-p-harmonic-main} after an easy calculation.
The second claim follows by noting that 
\begin{equation*}
\tau^r\lrpar{\prod_{k=1}^m \phi_k} = \sum_{j_1+\ldots+j_m = r} \binom{r}{j_1,\ldots, j_m} \, \prod_{k=1}^m \tau^{j_k}(\phi_k),
\end{equation*}
which can be proven analogously to Proposition \ref{prop-separation-1}.
\end{proof}

The matrix $B$ in part (i) of our Proposition \ref{prop-pharmonic-x-t} can of course be taken to be zero.
The following examples show though that in important cases this is not necessary.

\begin{example}\label{ex-g44}
Consider the four-dimensional Lie group 
\begin{equation*}
\mathrm{G}_{4.4} = \big\{ \begin{bmatrix}
e^{t} & te^{t} & \frac{1}{2}t^2e^{t} & x_1\\
0 & e^{t} & te^{t} & x_2\\
0 & 0 & e^{t} & x_3\\
0 & 0 & 0 & 1
\end{bmatrix} \,\mid\, (t,x) \in \rn \times \rn^3 \big\}.
\end{equation*}
This is the semidirect product
$\rn \ltimes_\A \rn^3$ with respect to the family $\A$ consisting of the single matrix
\begin{equation*}
\begin{bmatrix}
1 & 1 & 0\\
0 & 1 & 1\\
0 & 0 & 1
\end{bmatrix}.
\end{equation*}
Then following Proposition \ref{prop-separation-1}, combined with Proposition \ref{prop-pharmonic-x-t}, the function 
\begin{equation*}
(t,x) \mapsto (c_1 t^{r-1}e^{-3 t} + c_{2}t^{r-1}) (a_1 + a_2 x_1 + a_3 x_2 + a_4 (x_2^2 - x_3^3 - 2x_1x_3))
\end{equation*}
is proper $r$-harmonic on $\mathrm{G}_{4.4}$ for any non-zero elements $c \in \cn^2$ and $a \in \cn^4$.
\end{example}

\begin{example}\label{ex-g48}
For $\alpha \in [-1,1]$ consider the following four-dimensional Lie group 
\begin{equation*}
\mathrm{G}_{4.8}^\alpha
=
\big\{
\begin{bmatrix}
e^{(1+\alpha) t} & -\frac{x_2}{2}e^{t} & \frac{x_1}{2}e^{\alpha t} & \xi\\[0.1cm]
0 & e^{t} & 0 & x_1\\
0 & 0 & e^{\alpha t} & x_2\\
0 & 0 & 0 & 1
\end{bmatrix}
\,\mid\, (t,\xi,x) \in \rn \times \rn \times \rn^2 \big\}.
\end{equation*}
This is the semidirect product
$\rn \ltimes_\A \mathrm{H}^3$ with respect to the family $\A$ consisting of the single matrix
\begin{equation*}
\begin{bmatrix}
1 & 0\\
0 & \alpha
\end{bmatrix}.
\end{equation*}
Then Proposition \ref{prop-separation-1}, in conjunction with Proposition \ref{prop-pharmonic-x-t},  implies that the function defined by
\begin{equation*}
(t,\xi,x) \mapsto \begin{cases}
(c_1 t^{2r-1} + c_{2} t^{2r-2}) (a_1 + a_2\xi + a_3 x_1 + a_4 x_2 + a_5 x_1 x_2) & \text{if }\;\alpha = -1,\\[0.2cm]
(c_1 t^{r-1}e^{2(1+\alpha) t} + c_{2}t^{r-1}) (a_1 + a_2\xi + a_3 x_1 + a_4 x_2 + a_5 x_1 x_2) & \text{if }\;\alpha \not= -1
\end{cases} 
\end{equation*}
is proper $r$-harmonic on $\mathrm{G}_{4.8}^\alpha$ for any non-zero elements $c \in \cn^2$ and $a \in \cn^5$.
\end{example}

We now proceed by constructing a non-trivial class of complex isoparametric functions on $\rn^m \ltimes_\A G$.
These functions will depend only on the variables $t$ and $x$ i.e.\ they will be independent of the variable $\xi$ if $G = \mathrm{H}^{2n+1}$.
Here we will use the well-known fact that the elements of a commuting family of matrices always possess a common eigenvector.

\begin{proposition}\label{lemma-phi}
Let $v$ be a common eigenvector of the commuting family $\A^\transp = (A_k^\transp)_{k}$
and $\lambda = (\lambda_1, \ldots, \lambda_m)$ be the vector consisting of the corresponding eigenvalues.
Then the complex-valued function $\phi : \rn^m \ltimes_\A G \to \cn$ defined by
\begin{equation*}
\phi(t,x) = e^{-\langle \lambda, t \rangle} \langle v,x \rangle
\end{equation*}
is complex isoparametric on $\rn^m \ltimes_\A G$ with
\begin{equation*}
\tau(\phi) = \langle \lambda, \lambda + \omega \rangle \,  \phi \quad \text{and} \quad
\kappa(\phi,\phi) = \langle \lambda, \lambda \rangle \, \phi^2 + \langle v, v \rangle.
\end{equation*}
\end{proposition}

\begin{remark}
Here, $\langle \cdot, \cdot \rangle$ denotes the complex bilinear product of vectors given by
\begin{equation*}
\langle v,w \rangle = \sum_{i} v_i w_i, \quad v,w\in \cn^n.
\end{equation*}
\end{remark}

\begin{proof}
We have
\begin{equation*}
\frac{\partial \phi}{\partial t_k} = -\lambda_k \, \phi, \quad \frac{\partial^2 \phi}{\partial t_k^2} = \lambda_k^2 \, \phi, \quad
\frac{\partial \phi}{\partial x_i} = e^{-\langle\lambda, t\rangle}v_i\ \ \text{and}\ \ \frac{\partial^2 \phi}{\partial x_i \partial x_j} = 0.
\end{equation*}
Thus,
\begin{equation*}
\tau(\phi) = \sum_k ( \lambda_k^2 + \omega_k \lambda_k ) \, \phi = \langle\lambda, \lambda + \omega \rangle \, \phi,
\end{equation*}
as well as
\begin{align*}
\kappa(\phi,\phi)
&= \sum_k \lambda_k^2 \,\phi^2 + e^{-2\langle \lambda, t\rangle} \sum_{ij}(\mu(t) \, \mu(t)^\transp)_{ij} v_iv_j\\[0.1cm]
&= \langle \lambda, \lambda\rangle\,\phi^2 + e^{-2\langle\lambda, t\rangle} \langle\mu(t)^\transp v, \mu(t)^\transp v\rangle\\[0.1cm]
&= \langle \lambda, \lambda\rangle\,\phi^2 + \langle v,v\rangle,
\end{align*}
where the final equality follows from the fact that $v$ is an eigenvector of
\begin{equation*}
\mu(t)^\transp = \Exp\lrpar{\textstyle\sum_k A_k^\transp t_k}
\end{equation*}
with eigenvalue $e^{\langle\lambda,t\rangle}$ by assumption.
\end{proof}

Having constructed complex isoparametric functions, we can now apply Theorem \ref{thm-p-harmonic-main} to generate examples of $r$-harmonic functions.
As one might suspect, the antiderivatives from Theorem \ref{thm-p-harmonic-main} can not be computed explicitly in general.
In what follows, we consider several specific cases where this is possible.

\begin{example}\label{ex:g_41}
The four-dimensional Lie group
\begin{equation*}
\mathrm{G}_{4.1} = \big\{ \begin{bmatrix}
1 & t & \frac{1}{2}t^2 & x_1\\
0 & 1 & t & x_2\\
0 & 0 & 1 & x_3\\
0 & 0 & 0 & 1
\end{bmatrix} \,\mid\, (t,x) \in \rn \times \rn^3 \big\}
\end{equation*}
is the semidirect product $\rn \ltimes_\A \rn^3$, where $\A$ consists of the single matrix
\begin{equation*}
A = \begin{bmatrix}
0 & 1 & 0\\
0 & 0 & 1\\
0 & 0 & 0
\end{bmatrix}.
\end{equation*}
All the eigenvalues of $A^\transp$ are $0$ and the eigenspace is the one-dimensional space spanned by the unit vector $e_3$.
Thus, the function $\phi$ from Proposition \ref{lemma-phi} is simply the coordinate function
$(t,x) \mapsto x_3$ satisfying 
\begin{equation*}
\tau(\phi) = 0 \quad\text{and}\quad \kappa(\phi,\phi) = 1.
\end{equation*}
A simple application of Theorem \ref{thm-p-harmonic-main} then shows that the function
\begin{equation*}
(t,x) \mapsto a_1 x_3^{2r-1} + a_2 x_3^{2r-2}
\end{equation*}
is proper $r$-harmonic on $\mathrm{G}_{4.1}$.
By construction this function also satisfies the condition from Proposition \ref{prop-separation-1} in that its Laplacian of any order is independent of $t$.
Hence, if $p,r,q$ are positive integers such that $r+q-1=p$,
Proposition \ref{prop-separation-1} combined with part (ii) Proposition \ref{prop-pharmonic-x-t} implies that the function
 defined by
\begin{equation*}
(t,x) \mapsto (c_1 t^{2r-1} + c_2 t^{2r-2})(a_1 x_3^{2q-1} + a_2 x_3^{2q-2})
\end{equation*}
is proper $p$-harmonic on $\mathrm{G}_{4.1}$ for any non-zero $a,c \in \cn^2$.
\end{example}

\begin{example}\label{ex-sol3}
The celebrated Thurston geometry $\mathbf{Sol}^3$ is the solvable Lie group
\begin{equation*}
\mathbf{Sol}^3 = \big\{ \begin{bmatrix}
e^t & 0 & x_1\\
0 & e^{-t} & x_2\\
0 & 0 & 1
\end{bmatrix} \,\mid\, (t,x) \in \rn \times \rn^2 \big\}.
\end{equation*}
Examples of proper $r$-harmonic functions on this Lie group have already been constructed in \cite{Gud-14} and \cite{Gud-Sif-2}.
We can view this Lie group as the semidirect product $\rn \ltimes_\A \rn^2$, where $\A$ consists of the matrix
\begin{equation*}
A = \begin{bmatrix}
1 & 0\\
0 & -1
\end{bmatrix}.
\end{equation*}
The eigenvalues of $A^\transp$ are $1$ and $-1$, and the corresponding eigenvectors are $e_1$ and $e_2$, respectively.
In view of Proposition \ref{lemma-phi} we get an isoparametric function for each eigenvector, namely
\begin{equation*}
\phi_1(t,x) = e^{-t}x_1 \quad\text{and}\quad \phi_2(t,x) = e^t x_2,
\end{equation*}
both of which satisfy
\begin{equation*}
\tau(\phi_i) = \phi_i \quad\text{and}\quad \kappa(\phi_i,\phi_i) = \phi_i^2+1. 
\end{equation*}
Applying Theorem \ref{thm-p-harmonic-main} to these isoparametric functions, we find that the functions
\begin{align*}
(t,x) &\mapsto c_1 \arsinh(e^{-t} x_1)^{2r-1} + c_2 \arsinh(e^{-t} x_1)^{2r-2},\\[0.1cm]
(t,x) &\mapsto c_1 \arsinh(e^{t} x_2)^{2r-1} + c_2 \arsinh(e^{t} x_2)^{2r-2}
\end{align*}
are proper $r$-harmonic on $\mathbf{Sol}^3$
for any non-zero $c \in \cn^2$.
\end{example}

The reader should note that in both Examples \ref{ex:g_41} and \ref{ex-sol3} the eigenvector $v$ is non-isotropic i.e.\ $\langle v, v \rangle \not= 0$.
This is of course expected since the eigenvector is real in both cases.
However, if the eigenvector $v$ turns out to be complex, then it could happen that it is isotropic.
In this case we see directly that the function $\phi$ from Proposition \ref{lemma-phi} becomes an eigenfunction cf.\ (\ref{eq-eigenfunction-def}).
In fact, a quick inspection of the proof of Proposition \ref{lemma-phi} shows that in this case the vector $\lambda$ can be taken arbitrarily,
rather than requiring it to consist of eigenvalues.

\begin{corollary}
\label{prop-phi-pharmonic-v-isotropic-1}
Let $v$ be a common eigenvector of the commuting family $\A^\transp = (A_k^\transp)_{k}$
and suppose that $v$ is isotropic.
Then for any complex vector $\nu \in \cn^m$ the function $\phi: \rn^m \ltimes_\A G \to \cn$ defined by
\begin{equation*}
\phi(t,x) = e^{-\langle \nu, t \rangle } \langle v, x \rangle
\end{equation*}
is an eigenfunction on $\rn^m \ltimes_\A G$ satisfying
\begin{equation*}
\tau(\phi) = \langle \nu, \nu + \omega \rangle \,  \phi \quad \text{and} \quad
\kappa(\phi,\phi) = \langle \nu, \nu \rangle \, \phi^2.
\end{equation*}
\end{corollary}

For this situation the antiderivatives from Theorem \ref{thm-p-harmonic-main} have already been computed in \cite{Gud-Sob-1}, cf.\ also (\ref{eq-p-harmonic-eigenfunction}).

\begin{example}
Consider the following interesting four-dimensional Lie group
\begin{equation*}
\mathrm{G}_{4.10}
=
\big\{ \begin{bmatrix}
e^{-t_1}\cos(t_2) & e^{-t_1}\sin(t_2) & x_1 & 0\\
-e^{-t_1}\sin(t_2) & e^{-t_1}\cos(t_2) & x_2 & 0\\
0 & 0 & 1 & 0\\
0 & 0 & 0 & e^{t_2}\\
\end{bmatrix} \,\mid\, (t,x) \in \rn^2 \times \rn^2 \big\}.
\end{equation*}
This is the semidirect product $\rn^2 \ltimes_\A \rn^2$, where the family $\A$ consists of the two commuting matrices
\begin{equation*}
\begin{bmatrix}
-1 & 0\\
0 & -1
\end{bmatrix}, \quad
\begin{bmatrix}
0 & -1\\
1 & 0
\end{bmatrix}.
\end{equation*}
Their two common eigenvectors are $v_\pm = (1, \pm \mathrm i)$, both of which are isotropic.
Hence, we see from Corollary \ref{prop-phi-pharmonic-v-isotropic-1} that for any $\nu = (\nu_1, \nu_2) \in \cn^2$, the functions
\begin{equation*}
\phi_\pm(t,x) = e^{-(\nu_1 t_1 + \nu_2 t_2)} (x_1 \pm \mathrm i x_2)
\end{equation*} 
are eigenfunctions on $\mathrm{G}_{4.10}$ with
\begin{equation*}
 \tau(\phi_\pm) = (\nu_1^2 - 2\nu_1 + \nu_2^2) \, \phi_\pm  \quad\text{and}\quad   \kappa(\phi_\pm, \phi_\pm) = (\nu_1^2 + \nu_2^2) \, \phi_\pm^2. \qedhere
\end{equation*}
\end{example}

\smallskip

In the case when $v$ is isotropic, we can take the real and the imaginary part of the function $\phi$ from Proposition \ref{lemma-phi}
to obtain even more isoparametric functions.

\begin{proposition}\label{lemma-phi-ReIm}
Let $v$ be a common eigenvector of the commuting family $\A^\transp = (A_k^\transp)_{k}$ 
and let $\lambda = (\lambda_1, \ldots, \lambda_m)$ be the vector consisting of the corresponding eigenvalues.
Further suppose that the eigenvector $v$ is isotropic and define the function $\phi: \rn^m \ltimes_\A G \to \cn$ by
\begin{equation*}
\phi(t,x) = e^{-\langle \Re\lambda,t \rangle} \langle v, x \rangle.
\end{equation*}
Then the real part $\phi_1 = \Re\phi$ and the imaginary part $\phi_2 = \Im\phi$ of $\phi$ are isoparametric functions on $\rn^m \ltimes_\A G$ with
\begin{align*}
\tau(\phi_i) &= \langle \Re\lambda,\Re\lambda + \omega\rangle \,\phi_i,\\[0.1cm]
\kappa(\phi_i, \phi_i) &= \langle \Re\lambda, \Re\lambda \rangle \, \phi_i^2 +  \frac{1}{2} \langle v,\overline{v} \rangle.
\end{align*}
\end{proposition}

\begin{proof}
The result follows by direct calculations and Proposition \ref{lemma-phi}.
\end{proof}

\begin{example}
For $\alpha \geq 0$, consider the four-dimensional Lie group $\mathrm{G}_{4.9}^\alpha$ given by 
\begin{equation*}
\begin{bmatrix}
e^{2\alpha t} & -\frac{x_2\cos(t)+x_1\sin(t)}{2}e^{\alpha t} & \frac{x_1\cos(t)-x_2\sin(t)}{2}e^{\alpha t} & \xi & 0\\[0.135cm]
0 & e^{\alpha t}\cos(t) & e^{\alpha t}\sin(t) & x_1 & 0\\
0 & -e^{\alpha t}\sin(t) & e^{\alpha t}\cos(t) & x_2 & 0\\
0 & 0 & 0 & 1 & 0\\
0 & 0 & 0 & 0 & e^t
\end{bmatrix}
\end{equation*}
where $(t,\xi,x) \in \rn \times \rn \times \rn^2$.
We can view this group as the semidirect product $\rn \ltimes_\A \mathrm{H}^3$ with respect to the family $\A$ consisting of the single matrix
\begin{equation*}
A = \begin{bmatrix}
\alpha & 1\\
-1 & \alpha
\end{bmatrix}.
\end{equation*}
The two eigenvectors of $A^\transp$ are $v_\pm = (1,\pm \mathrm i)$ and the corresponding eigenvalues are ${\lambda_\pm = \alpha \pm \mathrm i}$.
Note that the eigenvectors $v_\pm$ are isotropic.
In what follows, it suffices to only consider the eigenvector $v = (1,\mathrm i)$ and its eigenvalue $\lambda = \alpha + \mathrm i$.
The functions $\phi_1$ and $\phi_2$ from Proposition \ref{lemma-phi-ReIm} are then given by
\begin{equation*}
\phi_i(t,\xi,x) = e^{-\alpha t} x_i,
\end{equation*}
satisfying
\begin{equation*}
\tau(\phi_i) = 3\alpha^2\phi_i \quad\text{and}\quad \kappa(\phi_i,\phi_i) = \alpha^2\phi_i^2 + 1.
\end{equation*}
We now apply Theorem \ref{thm-p-harmonic-main} to these isoparametric functions.

\smallskip

In the case when $\alpha = 0$ an easy calculation of the antiderivatives from Theorem \ref{thm-p-harmonic-main} shows that 
for any positive integer $r$ and any non-zero $c\in \cn^2$, the functions
\begin{equation*}
(t,\xi,x) \mapsto c_1 x_i^{2r-1} + c_2 x_i^{2r-2}, \quad i = 1, 2
\end{equation*}
are proper $r$-harmonic on $\mathrm{G}_{4.9}^0$.
By construction, these functions also satisfy the condition from Proposition \ref{prop-separation-1} 
so we may also multiply them by a function depending only on $t$ to obtain even more examples of proper $r$-harmonic functions
cf.\ Example \ref{ex:g_41}.

In the case when $\alpha \not= 0$ the functions $f_p$ from Theorem \ref{thm-p-harmonic-main} do not seem to possess a nice closed formula.
However, let us at least note that for any non-zero $c \in \cn^2$ the functions given by
\begin{equation*}
(t,\xi,x) \mapsto c_1 \, \frac{2(\alpha e^{-\alpha t} x_i)^3+3\alpha e^{-\alpha t} x_i}{\lrcurl{(\alpha e^{-\alpha t} x_i)^2+1}^{3/2}} + c_2, \quad i = 1, 2
\end{equation*}
are proper harmonic on $\mathrm{G}_{4.9}^\alpha$
and the functions given by
\begin{align*}
&(t,\xi,x) \mapsto c_1 \lrpar{\arsinh(\alpha e^{-\alpha t} x_i) + \frac{(\alpha e^{-\alpha t} x_i)^3}{3\lrcurl{(\alpha e^{-\alpha t} x_i)^2+1}^{3/2}}}\\[0.1cm]
&\quad +c_2 \lrpar{ \frac{ 2(\alpha e^{-\alpha t} x_i)^3 + 3\alpha e^{-\alpha t} x_i }{\lrcurl{(\alpha e^{-\alpha t} x_i)^2+1}^{3/2}} \, \arsinh(\alpha e^{-\alpha t} x_i)
- \frac{1}{(\alpha e^{-\alpha t} x_i)^2+1}}, \quad i = 1, 2
\end{align*}
are proper biharmonic on $\mathrm{G}_{4.9}^\alpha$.
\end{example}

\vspace*{\fill}

\appendix

\section{Tables}

\begin{table}[H]
\centering
\small

\renewcommand{\arraystretch}{1.5}

\makebox[\textwidth][c]{
\begin{tabular}{lllll}
\toprule
Lie algebra & Parameters & \multicolumn{3}{l}{Lie brackets, basis $\{E_1,E_2,E_3,E_4\}$} \\ \midrule
\multirow{2}{*}{$\g_{4.1}$} & &
$\lb{E_4}{E_1} = 0$, & $\lb{E_4}{E_2} = E_1$, & $\lb{E_4}{E_3} = E_2$, \\[-0.1cm]
& &
$\lb{E_3}{E_1} = 0$, & $\lb{E_3}{E_2} = 0$, & $\lb{E_2}{E_1} = 0$ \\[0.25cm]
\multirow{2}{*}{$\g_{4.2}^\alpha$} & \multirow{2}{*}{$\alpha \not= 0$} &
$\lb{E_4}{E_1} = \alpha E_1$, & $\lb{E_4}{E_2} = E_2$, & $\lb{E_4}{E_3} = E_2+E_3$, \\[-0.1cm]
& &
$\lb{E_3}{E_1} = 0$, & $\lb{E_3}{E_2} = 0$, & $\lb{E_2}{E_1} = 0$ \\[0.25cm]
\multirow{2}{*}{$\g_{4.3}$} & &
$\lb{E_4}{E_1} = E_1$, & $\lb{E_4}{E_2} = 0$, & $\lb{E_4}{E_3} = E_2$, \\[-0.1cm]
& &
$\lb{E_3}{E_1} = 0$, & $\lb{E_3}{E_2} = 0$, & $\lb{E_2}{E_1} = 0$ \\[0.25cm]
\multirow{2}{*}{$\g_{4.4}$} & &
$\lb{E_4}{E_1} = E_1$, & $\lb{E_4}{E_2} = E_1+E_2$, & $\lb{E_4}{E_3} = E_2+E_3$, \\[-0.1cm]
& &
$\lb{E_3}{E_1} = 0$, & $\lb{E_3}{E_2} = 0$, & $\lb{E_2}{E_1} = 0$ \\[0.25cm]
\multirow{2}{*}{$\g_{4.5}^{\alpha\beta\gamma}$} & \multirow{2}{*}{$\alpha\beta\gamma \not= 0$} &
$\lb{E_4}{E_1} = \alpha E_1$, & $\lb{E_4}{E_2} = \beta E_2$, & $\lb{E_4}{E_3} = \gamma E_3$, \\[-0.1cm]
& &
$\lb{E_3}{E_1} = 0$, & $\lb{E_3}{E_2} = 0$, & $\lb{E_2}{E_1} = 0$ \\[0.25cm]
\multirow{2}{*}{$\g_{4.6}^{\alpha\beta}$} & $\alpha > 0$, &
$\lb{E_4}{E_1} = \alpha E_1$, & $\lb{E_4}{E_2} = \beta E_2 - E_3$, & $\lb{E_4}{E_3} = E_2 + \beta E_3$, \\[-0.1cm]
& $\beta \in \rn$ &
$\lb{E_3}{E_1} = 0$, & $\lb{E_3}{E_2} = 0$, & $\lb{E_2}{E_1} = 0$ \\[0.25cm]
\multirow{2}{*}{$\g_{4.7}$} & &
$\lb{E_4}{E_1} = 2E_1$, & $\lb{E_4}{E_2} = E_2$, & $\lb{E_4}{E_3} = E_2+E_3$, \\[-0.1cm]
& &
$\lb{E_3}{E_1} = 0$, & $\lb{E_3}{E_2} = -E_1$, & $\lb{E_2}{E_1} = 0$ \\[0.25cm]
\multirow{2}{*}{$\g_{4.8}^\alpha$} & \multirow{2}{*}{$\alpha \in [-1,1]$} &
$\lb{E_4}{E_1} = (1+\alpha) E_1$, & $\lb{E_4}{E_2} = E_2$, & $\lb{E_4}{E_3} = \alpha E_3$, \\[-0.1cm]
& &
$\lb{E_3}{E_1} = 0$, & $\lb{E_3}{E_2} = -E_1$, & $\lb{E_2}{E_1} = 0$ \\[0.25cm]
\multirow{2}{*}{$\g_{4.9}^\alpha$} & \multirow{2}{*}{$\alpha \geq 0$} &
$\lb{E_4}{E_1} = 2\alpha E_1$, & $\lb{E_4}{E_2} = \alpha E_2 - E_3$, & $\lb{E_4}{E_3} = E_2 + \alpha E_3$, \\[-0.1cm]
& &
$\lb{E_3}{E_1} = 0$, & $\lb{E_3}{E_2} = -E_1$, & $\lb{E_2}{E_1} = 0$ \\[0.25cm]
\multirow{2}{*}{$\g_{4.10}$} & &
$\lb{E_4}{E_1} = -E_2$, & $\lb{E_4}{E_2} = E_1$, & $\lb{E_4}{E_3} = 0$, \\[-0.1cm]
& &
$\lb{E_3}{E_1} = -E_1$, & $\lb{E_3}{E_2} = -E_2$, & $\lb{E_2}{E_1} = 0$ \\[0.1cm]
\bottomrule
\end{tabular}
}
\vskip0.3cm

\caption{\cite{Pop-Boy-etal,Big-Rem} The (indecomposable) four-dimensional Lie algebras}
\label{table:4dimalgebras}
\end{table}

\vspace*{\fill}

\newpage

\vspace*{\fill}

\begin{table}[H]
\centering
\footnotesize

\makebox[\textwidth][c]{
\begin{tabular}{llcc}
\toprule
Lie group & Parameters & Family $\A$ & Linear representation \\\midrule
$\mathrm{G}_{4.1} = \rn \ltimes_\A \rn^3$ & &
$
\begin{bmatrix}
0 & 1 & 0\\
0 & 0 & 1\\
0 & 0 & 0
\end{bmatrix}
$ &
$\begin{bmatrix}
1 & t & \frac{1}{2}t^2 & x_1\\
0 & 1 & t & x_2\\
0 & 0 & 1 & x_3\\
0 & 0 & 0 & 1
\end{bmatrix}
$ \medskip \\
$\mathrm{G}_{4.2}^\alpha = \rn \ltimes_\A \rn^3$ & $\alpha \not= 0$ &
$
\begin{bmatrix}
\alpha & 0 & 0\\
0 & 1 & 1\\
0 & 0 & 1
\end{bmatrix}
$ &
$\begin{bmatrix}
e^{\alpha t} & 0 & 0 & x_1\\
0 & e^{t} & te^{t} & x_2\\
0 & 0 & e^{t} & x_3\\
0 & 0 & 0 & 1
\end{bmatrix}
$ \medskip \\
$\mathrm{G}_{4.3} = \rn \ltimes_\A \rn^3$ & &
$
\begin{bmatrix}
1 & 0 & 0\\
0 & 0 & 1\\
0 & 0 & 0
\end{bmatrix}
$ &
$
\begin{bmatrix}
e^{t} & 0 & 0 & x_1\\
0 & 1 & t & x_2\\
0 & 0 & 1 & x_3\\
0 & 0 & 0 & 1
\end{bmatrix}
$
\medskip \\
$\mathrm{G}_{4.4} = \rn \ltimes_\A \rn^3$ & &
$
\begin{bmatrix}
1 & 1 & 0\\
0 & 1 & 1\\
0 & 0 & 1
\end{bmatrix}
$ &
$
\begin{bmatrix}
e^{t} & te^{t} & \frac{1}{2}t^2e^{t} & x_1\\
0 & e^{t} & te^{t} & x_2\\
0 & 0 & e^{t} & x_3\\
0 & 0 & 0 & 1
\end{bmatrix}
$ \medskip \\
$\mathrm{G}_{4.5}^{\alpha\beta\gamma} = \rn \ltimes_\A \rn^3$ & $\alpha\beta\gamma \not= 0$ &
$
\begin{bmatrix}
\alpha & 0 & 0\\
0 & \beta & 0\\
0 & 0 & \gamma
\end{bmatrix}
$ &
$
\begin{bmatrix}
e^{\alpha t} & 0 & 0 & x_1\\
0 & e^{\beta t} & 0 & x_2\\
0 & 0 & e^{\gamma t} & x_3\\
0 & 0 & 0 & 1
\end{bmatrix}
$ \medskip \\
$\mathrm{G}_{4.6}^{\alpha\beta} = \rn \ltimes_\A \rn^3$ & $\alpha > 0, \beta \in \rn$ &
$
\begin{bmatrix}
\alpha & 0 & 0\\
0 & \beta & 1\\
0 & -1 & \beta
\end{bmatrix}
$ &
$
\begin{bmatrix}
e^{\alpha t} & 0 & 0 & x_1\\
0 & e^{\beta t}\cos(t) & e^{\beta t}\sin(t) & x_2\\
0 & -e^{\beta t}\sin(t) & e^{\beta t}\cos(t) & x_3\\
0 & 0 & 0 & 1
\end{bmatrix}
$ \medskip \\
$\mathrm{G}_{4.7} = \rn \ltimes_\A \mathrm{H}^3$ & &
$
\begin{bmatrix}
1 & 1\\
0 & 1
\end{bmatrix}
$ &
$\begin{bmatrix}
e^{2t} & -\frac{x_2}{2}e^{t} & \frac{x_1-tx_2}{2}e^t & \xi\\[0.05cm]
0 & e^t & te^t & x_1\\
0 & 0 & e^t & x_2\\
0 & 0 & 0 & 1
\end{bmatrix}
$ \medskip \\
$\mathrm{G}_{4.8}^\alpha = \rn \ltimes_\A \mathrm{H}^3$ & $\alpha \in [-1,1]$ &
$
\begin{bmatrix}
1 & 0\\
0 & \alpha
\end{bmatrix}
$ &
$\begin{bmatrix}
e^{(1+\alpha) t} & -\frac{x_2}{2}e^{t} & \frac{x_1}{2}e^{\alpha t} & \xi\\[0.075cm]
0 & e^{t} & 0 & x_1\\
0 & 0 & e^{\alpha t} & x_2\\
0 & 0 & 0 & 1
\end{bmatrix}
$ \medskip \\
$\mathrm{G}_{4.9}^\alpha = \rn \ltimes_\A \mathrm{H}^3$ & $\alpha \geq 0$ &
$
\begin{bmatrix}
\alpha & 1\\
-1 & \alpha
\end{bmatrix}
$ &
$
\begin{bmatrix}
e^{2\alpha t} & -\frac{x_2\cos(t)+x_1\sin(t)}{2}e^{\alpha t} & \frac{x_1\cos(t)-x_2\sin(t)}{2}e^{\alpha t} & \xi & 0\\[0.075cm]
0 & e^{\alpha t}\cos(t) & e^{\alpha t}\sin(t) & x_1 & 0\\
0 & -e^{\alpha t}\sin(t) & e^{\alpha t}\cos(t) & x_2 & 0\\
0 & 0 & 0 & 1 & 0\\
0 & 0 & 0 & 0 & e^t
\end{bmatrix}
$ \medskip \\
$\mathrm{G}_{4.10} = \rn^2 \ltimes_\A \rn^2$ &  &
$
\begin{bmatrix}
-1 & 0\\
0 & -1
\end{bmatrix},\;
\begin{bmatrix}
0 & 1\\
-1 & 0
\end{bmatrix}
$ &
$
\begin{bmatrix}
e^{-t_1}\cos(t_2) & e^{-t_1}\sin(t_2) & x_1 & 0\\
-e^{-t_1}\sin(t_2) & e^{-t_1}\cos(t_2) & x_2 & 0\\
0 & 0 & 1 & 0\\
0 & 0 & 0 & e^{t_2}\\
\end{bmatrix}
$ \\ \bottomrule
\end{tabular}
}

\vskip0.4cm

\caption{\cite{Pop-Boy-etal,Big-Rem} Linear representations of the four-dimensional Lie groups}
\label{table-groups-4dim}
\end{table}

\vspace*{\fill}


\begin{thebibliography}{99}

\bibitem{Bai} 
P. Baird,
{\it Conformal foliations by circles and complex isoparametric functions on Euclidean 3-space},
Math. Proc. Cambridge Philos. Soc. {\bf 123} (1998), 273-300. 

\bibitem{Bai-Far-Oua} 
P. Baird, A. Fardoun, S. Ouakkas,
{\it Biharmonic maps from biconformal transformations with respect to isoparametric functions},
Differential Geom. Appl. {\bf 50} (2017), 155-166. 

\bibitem{Big-Rem}
R. Biggs, C. Remsing,
{\it On the Classification of Real Four-Dimensional Lie Groups},
J. Lie Theory {\bf 26} (2016), 1001-1035.

\bibitem{Gud-13}
S. Gudmundsson,
{\it Biharmonic functions on the special unitary group $\SU 2$},
Differential Geom. Appl. {\bf 53} (2017), 137-147.

\bibitem{Gud-14}
S. Gudmundsson,
{\it A note on biharmonic functions on the Thurston geometries},
J. Geom. Phys. {\bf 131} (2018), 114-121.

\bibitem{Gud-15}
S. Gudmundsson,
{\it Biharmonic functions on spheres and hyperbolic spaces},
J. Geom. Phys. {\bf 134} (2018), 244-248.

\bibitem{Gud-Mon-Rat-1}
S. Gudmundsson, S. Montaldo, A. Ratto,
{\it Biharmonic functions on the classical compact simple Lie groups},
J. Geom. Anal. {\bf 28} (2018), 1525-1547.

\bibitem{Gud-Sif-1}
S. Gudmundsson, A. Siffert,
{\it New biharmonic functions on the compact Lie groups $\SO n$, $\SU n$, $\Sp n$}, 
J. Geom. Anal. (2019) - \href{https://doi.org/10.1007/s12220-019-00259-3}{doi.org/10.1007/s12220-019-00259-3}

\bibitem{Gud-Sif-2}
S. Gudmundsson, A. Siffert,
{\it Proper r-harmonic functions on the Thuston geometries},
preprint (2019). \href{https://arxiv.org/abs/1910.13477}{arXiv:1910.13477 [math.DG]}

\bibitem{Gud-Sob-1}
S. Gudmundsson, M. Sobak,
{\it Proper $r$-harmonic functions from Riemannian manifolds},
Ann. Global Anal. Geom. {\bf 57} (2020), 217-223.

\bibitem{Mel}
V. Meleshko,
{\it Selected topics in the history of the two-dimensional biharmonic problem},
Appl. Mech. Rev. {\bf 56} (2003), 33-85.

\bibitem{Pop-Boy-etal}
R. O. Popovych, V. M. Boyko, M. O. Nestrenko, M. W. Lutfullin,
{\it Realizations of real low-dimensional Lie algebras},
J. Phys. A. {\bf 36} (2003), 7337-7360.

\bibitem{Sob-MSc}
M. Sobak,
{\it $p$-Harmonic Functions on Riemannian Lie Groups}, Master's thesis, 
Lund University (2020).
\href{http://www.matematik.lu.se/matematiklu/personal/sigma/students/Marko-Sobak-MSc.pdf}{http://www.matematik.lu.se/matematiklu/personal/sigma/students/Marko-Sobak-MSc.pdf}

\bibitem{Sno-Win}
L. {\v S}nobl, P. Winternitz,
{\it Classification and Identification of Lie Algebras}, 
CRM Monograph Series {\bf 33},
American Mathematical Society (2014).

\bibitem{Tho}
G. Thorbergsson,
{\it A Survey on Isoparametric Hypersurfaces and Their Generalizations}, 
In: Handbook of Differential Geometry {\bf 1} (2000), 963-995. 

\end{thebibliography}
\end{document}